\documentclass{article}

\usepackage{graphicx}
\usepackage{amsmath}
\usepackage{amsthm}
\usepackage{amssymb}
\usepackage{verbatim}
\usepackage{epsfig}
\usepackage{wrapfig}

\theoremstyle{definition}
\newtheorem{definition}{Definition}[section]
\newtheorem{notation}{Notation}[section]

\theoremstyle{remark}
\newtheorem{example}[definition]{Example}
\newtheorem{remark}[definition]{Remark}

\theoremstyle{plain}
\newtheorem{lemma}[definition]{Lemma}
\newtheorem{proposition}[definition]{Proposition}
\newtheorem{theorem}[definition]{Theorem}
\newtheorem{corollary}[definition]{Corollary}

\newtheorem{conjecture}[definition]{Conjecture}

\newcommand{\bo}{\mathbf}

\newcommand{\IQ}{\mathbb{Q}}

\newcommand{\IP}{\mathbb{P}}

\renewcommand{\bold}{\textbf}
\renewcommand{\bar}{\overline}

\newcommand{\mo}{\bar{\bo{M}}_{0,n}}
\newcommand{\wo}{\bar{\bo{M}}_{0,\mathcal{A}}}

\newcommand{\A}{\mathcal{A}}
\newcommand{\plus}{\bigoplus}
\newcommand{\floor}{\lfloor{n/2}\rfloor}
\newcommand{\tensor}{\otimes}
\newcommand{\inject}{\hookrightarrow}
\newcommand{\intersect}{\cap}
\begin{document}

\title{On Log Canonical Models of the Moduli Space of Stable Pointed Curves.}
\author{Matthew Simpson}
\maketitle

We study the log canonical models of the moduli space $\mo$ of
pointed stable genus zero curves associated to the divisors of the
form $K_{\mo}+\alpha \text{D}$ where $D$ denotes the boundary of
$\mo$.  In particular we will show that, as a formal consequence of
a conjecture by Fulton regarding the ample cone of $\mo$, the log
canonical model associated to $K_{\mo}+\alpha \text{D}$ is equal to
the moduli space of weighted stable curves
with symmetric weights dependant on $\alpha$.

\section{Introduction}
\quad\\

Historically, an important problem related to the minimal model
program is to find the canonical model for $\bar{\bo{M}}_g$, the
moduli space of stable genus $g$ curves.  Here, the canonical model
(if it exists) is the projective scheme associated to the graded
ring (called the canonical ring) of sections $\plus_{n\geq
0}\Gamma(nK_{\bar{\bo{M}}_g})$. $\bar{M}_g$ has been proven to be of
general type for $g\geq 22$ [EH,Fa,HaMu].  In particular in this
case the finite generation of the canonical ring implies that the
canonical model is birational to $\bar{M}_g$. Recent progress has
been made in this area. For example [BCHM] have proved the existence
of canonical models in the case of smooth projective varieties of
general type.

 A somewhat easier problem is to study the \emph{log} canonical model associated to the graded ring {
$\plus_{n\geq 0}\Gamma(n(K_{\bar{\bo{M}}_g}+\alpha {D}))$
\normalsize}where $D$ is the boundary component parameterizing
singular curves and $\alpha$ is a rational number. This ring is
known to be finitely generated for $7/10\leq \alpha\leq
1$ and the models have been explicitly constructed using geometric invariant theory for certain values of $\alpha$ [CoHa,HH].\\

We play a similar game with the related but simpler space, the
moduli space of n-pointed stable genus zero curves $\mo$.  As usual
let $D$ denote the boundary component (a normal crossing divisor) of
$\mo$. We are interested in the finite generation of the ring
$R=\plus_{n\geq 0}\Gamma(n(K_{\mo}+\alpha \text{D}))$ of sections
associated to the divisor $K_{\mo}+\alpha D$. By Kawamata's
basepoint freeness theorem (below), the associated sheaf to $n(K_{\mo}+\alpha
\text{D})$ is generated by its global sections, for divisible
$n>>0$ and $0<\alpha<1$, if $K_{\mo}+\alpha D$ is nef and big. Thus
$R$ is finitely generated in such cases.  For small $\alpha$,
however, $K_{\mo}+\alpha D$ is not nef. To get around this we will
try to find nef and big divisors which have the same ring of
sections as our log canonical rings. Moreover we will try to do this
in such a way that we get an explicit description of our log
canonical models
in addition to existence.\\

In the next two sections we will relate the section rings of
$K_{\mo}+\alpha D$ to certain divisors on the moduli spaces of
weighted stable pointed curves.  In particular positivity of these
new divisors will imply that the weighted stable curve space is our
log canonical model.  In section $4$ we will state the main theorem and show that Fulton's
conjecture implies the positivity results which we need and section
$5$ will then survey the cases in which the result can be proved
directly without Fulton's conjecture. Finally, in section $6$, we
examine the neighborhood of the $\mathfrak{s}_n$-equivariant ample cone on
$\mo$ in which the divisors we are interested in live and show that
it (conjecturally) has a very simple structure.\\

\bold{Acknowledgments:} "This material is based upon work supported by the National Science
Foundation
under VIGRE Grant No. 0240058."

\section{Preliminaries}
We first recall some general facts.

\begin{theorem}[Kawamata basepoint freeness]([KM], Theorem 3.3)\\
Let $(X,D)$ be a proper pair with klt singularities such that $D$ is effective. Suppose A is a nef cartier
divisor such that $A-\epsilon(K_X+D)$ is nef and big for some
$\epsilon>0.$  Then $|nA|$ is basepoint free for $n>>0$.
\end{theorem}

In particular we may have $A=K_X+D$.\\

\begin{lemma}\label{one}
\text{Let $f: X\to Y$ be a proper morphism such}\text{ that $f_*
O_X=O_Y$}\\ \text{and let $L$ be a locally free sheaf on $Y$.}
\text{Then we have an isomorphism of global }\\\text{sections
$\Gamma(X,f^* L)=\Gamma(Y,L).$}

\end{lemma}

\begin{proof}
We have $\Gamma(X,f^*L)=\Gamma(Y,f_* f^* L)$ which, by the
projection formula, is isomorphic to $\Gamma(Y, f_* O_X\tensor L)$.
By assumption this is $\Gamma(Y,O_Y \tensor L)=\Gamma(Y, L)$.

\end{proof}

We will be interested in the following application.

\begin{proposition}\label{two}
Let $f:X\to Y$ be a birational morphism between normal projective varieties.  Let $D$ be a divisor on $Y$ and $F$ be an
effective divisor on $X$ whose support is contracted by $f$. Then
$\Gamma(X,f^*D+F)=\Gamma(Y,D)$.

\end{proposition}
Note: Throughout the paper $f^*D$ will be a shorthand for $f^* O_X(D)$.

\begin{proof}
By the lemma we have $\Gamma(Y,D)=\Gamma(X,f^* D)$.  Let $U$ be the
set on which $f$ is an isomorphism. We note that the
$codim(Y,Y/U)\geq 2.$  In particular this implies that sections of
$D$ over $U$ lift to sections of $D$ over $Y$, so
$\Gamma(Y,D)=\Gamma(U,D_{|U})$. Putting this all together we get a
sequence of morphisms
$$\Gamma(X,f^*D)\inject \Gamma(X,f^*D+F)\inject
\Gamma(U,(f^*D+F)_{|U}))$$$$=\Gamma(U,f^*
D_{|U})=\Gamma(U,D_{|U})=\Gamma(Y,D)=\Gamma(X,f^* D).$$

The first isomorphism is from the fact that F is supported away from
U and the injections are the natural ones gotten by inclusion and
restriction.  Finally, these are all finite dimensional vector
spaces so the fact that the total string of injective maps makes an
isomorphisms implies that each individual map is an isomorphism.

\end{proof}

Finally we note that if $D$ is ample in the above situation then
proposition \ref{two} implies that  $$Proj \plus_{n\geq
0}\Gamma(X,n(f^*D+F))=Proj\plus_{n\geq 0}\Gamma(Y,nD)=Y$$

\section{Weighted Stable Curves}
$K_{\mo}+\alpha D$ can be written in a nice way in terms of the boundary divisors $D_j$, $2\leq j\leq\floor$, parameterizing nodal curves with a node separating $j$ marked points
from the others.  In particular we have:
\begin{lemma}([Pan], Proposition 1)
$$K_{\mo}=\sum_{j=2}^{\floor}D_j\left(\frac{(j-2)(n-1)-j(j-1)}{n-1}\right), D=\sum_{j=2}^{\floor} D_j.$$
\end{lemma}
This implies that $K_{\mo}+\alpha D$ is effective (big) when
$\alpha\geq\frac{2}{n-1}(>)$.  [KeM], Lemma 4.8 implies that the inequality for bigness is tight.  Moreover since D is a normal crossing
divisor on a smooth variety this implies by definition of log
canonical/terminal that $K_{\mo}+\alpha D$ is log canonical
(terminal) for $\alpha\leq 1(<1).$ From now on we will assume
$\alpha$ is a rational number in the interval of interest,
$$\frac{2}{n-1}<\alpha\leq 1.$$

\begin{definition}
For a given $n$ and a given number $\alpha\in (\frac{2}{n-1},1] $,
let $k$ be the largest integer in the range $1,\dots,\lfloor
n/2\rfloor$ such that $\alpha\leq2/(k+1)$.  We define the following
$\mathfrak{s}_n$-equivariant divisor on $\mo$:
$$A_{\alpha}:=K_{\mo}+\sum_{j=2}^{k}D_j\left( \genfrac(){0cm}{0}j 2 \alpha
-(j-2)\right)+\sum_{j=k+1}^{\lfloor n/2\rfloor}\alpha D_j.$$
\end{definition}
These will be the divisors we are interested in.  As we will see shortly, they are pullbacks of nice log canonical divisors on certain related moduli spaces.  It should also be noted that it makes since to allow $\alpha=\frac{2}{n-1}$ but it is uninteresting as $A_{\frac{2}{n-1}}=0$.\\

Let $\wo$ be the moduli space of weighted stable pointed curves
associated to $n$ points with symmetric weight $\A=\{a,a,\ldots,a\}$
with $\frac{1}{k+1}<a\leq \frac{1}{k}.$ This space is smooth and
projective. In his introduction to these spaces, Hassett [Ha] showed
that there is a natural birational morphism $\rho:\mo\rightarrow
\wo$. In particular this map, on the level of pointed stable genus
zero curves, acts on curves by contracting components which become
unstable with the new weights. i.e. if a component has $b$ nodes and
$c$ marked points then it will be contracted unless $ac+b>2$.  Our
map $\rho$ is an isomorphism away from the boundary $D$ and
contracts precisely the boundary components $D_i$ with $i\neq 2$ and
$i a \leq 1$.\\

From the properties of the map $\rho$ described in [Ha], section 4,
we have the following:
\begin{lemma}

The map $\rho$ factors into a sequence of blowups along smooth
centers
$$\rho:\mo\overset{B_3}{\rightarrow} {\bar{\bo{M}}}_{0,A(3)}
\overset{B_4}{\rightarrow} \bar{\bo{M}}_{0,A(4)}\rightarrow \dots
\overset{B_k}{\rightarrow} \bar{\bo{M}}_{0,A(k )}=\wo$$ where map
$B_i$ blows down the image of $D_i$ under the map $B_{i-1}\circ
B_{i-2}\circ\dots\circ B_{3}$. If $k<3$, $\rho$ is an isomorphism.
Here $A(i)=\{a,a,\dots,a\}$ with $\frac{1}{i+1}<a\leq \frac{1}{i}.$

\end{lemma}

\begin{theorem}
Let $E$ denote the boundary of $\wo$.  Then $\rho^{*}(K_{\wo}+\alpha
E)=A_{\alpha}$ and $K_{\mo}+\alpha D-A_{\alpha}$ is effective and
$\rho$ -exceptional.
\end{theorem}

\begin{proof}
Let $D_j^i$ be the image of $D_j$ under the map $B_{i-1}\circ
B_{i-2}\circ\dots\circ B_{3}$.\\
  $\rho$ contracts precisely $D_3
,\dots , D_k $ and is birational on $D_2,D_{k+1},\dots D_{\floor}$
so we have a discrepancy equation of the form
$$K_{\mo}=-{{\rho}_*}^{-1}(\alpha E)+\rho^{*}(K_{\wo}+\alpha
E)+\sum_{i=3}^k a_i D_i$$
$$= -\alpha(D_2+D_{k+1}+\dots
D_{\floor})+\rho^{*}(K_{\bar{\bo{M}}_{0,n,\alpha}}+\alpha
E)+\sum_{i=3}^k a_i D_i.$$

We can use our factorization of $\rho$ to compute the $a_i$ via the
discrepancy formula for a blow-up.  Suppose we've already done this
for $i+1\leq j \leq k$.  $M_{0,A(i-1)} \overset{B_i}{\rightarrow}
M_{0,A(i)}$ is the blow-up map along $D_i^i.$  Now we know that
$D_i$ is canonically isomorphic to a finite union of copies of
$\bar{M}_{0,i+1}\times\bar{M}_{0,n-i+1}$ and the image of $\rho$ is
birational to  a union of copies of $\bar{M}_{0,n-i+1}$.  In
particular, $D_i^i$ has codimension $(i+1)-3+1=i-1$.  Therefore,
calculating the discrepancy from the blow-up maps, we get
$$a_i=(i-1)-1-\sum_{j=2,i+1,i+2\dots \floor}{{\beta}_j}mult_{D_i^i}D_j^i$$ where ${\beta}_j$ equals $a_i$ for $i+1\leq j \leq k$ and $\alpha$ otherwise.  That is to say they are the parts of the discrepancy formula that we have already calculated.

We break symmetry in order to calculate the needed multiplicities.
Let $\delta_S$, where $S$ is a subset of $\{1,\dots,n\}$, represent
the standard divisors of pointed curves separating $S$ and $S^c$ by
a node.  Let $S=\{1,\dots, i\}$ and $T$ be a subset of
$\{1,\dots,n\}$ of order $j$.  We can calculate $mult_{D_i^i}D_j^i$
by computing multiplicities on the images of $\delta_S$ and the
${\delta_T}'s$.

By Keel's description of the intersection ring in [Ke], $\delta_S$
and $\delta_T$ intersect iff $S\subset S'$ or $S\subset T'$ or
$T\subset S'$ or $T\subset T'$. $mult_{\delta_S}\delta_T$ equals one
in cases in which the two divisors intersect.

In the cases that $S\subset T$ or $S\subset T^c$ the intersection
$\delta_S\intersect \delta_T$  will be contracted by our map
$B_{i-1}\circ B_{i-2}\circ\dots\circ B_{3}$.  The corresponding
multiplicity $mult_{\text{image}(\delta_S})\text{image} (\delta_T)$
becomes zero in this case.

The other possibility for intersection is $S^c\subset T$ or
$S^c\subset T^c.$  In either of these cases we are forced to have
less elements in $T$ than in $S$.  However the order $j$ of $T$ is
either $2$ or $\geq i+1$ by assumption.  Therefore we must have
$j=2.$  In this case the intersection does not get contracted, and
the multiplicity of the images of the two divisors is still one.
 To go back to the symmetric case we observe that  there are $\genfrac(){0cm}{0}j 2$ divisors $\delta_{\{a,b\}}$ where
$\{a,b\}\subset\{1,\dots,j\}.$ Thus the multiplicity of $\delta_S$
with $D_2$ is $\genfrac(){0cm}{0}j 2$. Since $S$ was arbitrary the
multiplicity of $D_{i}$ with $D_{2}$ is the same and as observed the
multiplicities of their images are also the same.  Now we have:

$$K_{\mo}=-\alpha(D_2+D_{k+1}+\dots D_{\floor})+\rho^{*}(K_{\bar{\bo{M}}_{0,n,\alpha}}+\alpha D)-\sum_{i=3}^k(\genfrac(){0cm}{0}j 2\alpha -(j-2)) D_i.$$
Rearranging the above equation gives us $A_{\alpha}$ as wished. Also
$K_{\mo}+\alpha D-A_{\alpha}=K_{\mo}+\alpha
D-(K_{\bo{M_{0,n}}}+\sum_{j=1}^{k}D_j\left(\text{$\genfrac(){0cm}{0}j 2$}
\alpha-(j-2)\right)+\sum_{j=k+1}^{\lfloor n/2\rfloor}\alpha
D_j=\sum_{i=3}^k ((j-2)+\alpha(1- \genfrac(){0cm}{0}j 2))$ which is
$\rho$-exceptional by construction and is effective since
$\alpha\leq \frac{2}{k+1}\leq\frac{2}{j+1}$, so that the
coefficients of our sum are greater than or equal to
$\frac{1}{j+1}\left((j-2)(j+1)+2-j(j-1)\right)=0.$
\end{proof}
\begin{corollary}\label{cor}
If $K_{\wo}+\alpha E$ is ample then $\wo$ is the log canonical model
of $\mo$ associated to the divisor $K_{\mo}+\alpha D.$  In
particular this will be true if $A_{\alpha}$ is nef and only
contracts $\rho$-exceptional curves.

\end{corollary}

\bold{Note:}  Technically $\wo$ is only defined for $na>2$.
Therefore for the smallest values of $\alpha$ the above process was
not well defined. However, as mentioned in [Ha], we can define
$\wo$ for $a\leq 2/n$ to be ${\IP}^1/SL_2$ with symmetric weight $a$.
There still exists a map $\rho: \mo\to {\IP}^1/SL_2$ which contracts
all the boundary divisors except for $D_2$.  The above computation
plays out unchanged and the pullback result is still true.\\

\section{F-divisors}

\begin{definition}
Let $D$ be a divisor on $\mo$.  We say that $D$ is \emph{F-nef} (or
Fulton nef) if it intersects all vital curves on $\mo$
nonnegatively, where by a vital curve we mean an irreducible
component of the closed set in $\mo$ corresponding to pointed genus
zero stable curve with at least n-4 nodes.
\end{definition}

\begin{figure}[h]
  \caption{The generic element of a vital curve in ${\bar{\bo{M}}_{0,10}}$}
  \begin{center}
    \includegraphics{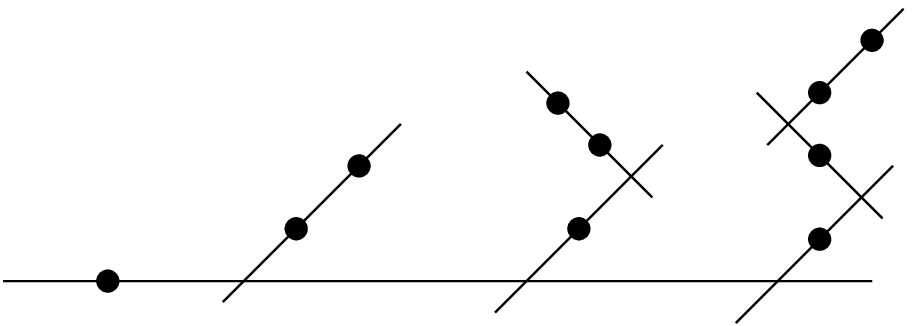}
  \end{center}
\end{figure}

We denote the closure of the real cone of effective curves up to
numerical equivalence by $\bar{NE}_1$.

 \begin{conjecture}{Fulton's Conjecture}\\

 $\bar{NE}_1$ is generated by vital curves.  In other words, a
 divisor is nef (ample) if and only if it intersects all vital curves
 nonnegatively (positively).

 \end{conjecture}

\begin{theorem}[Farkas, Gibney][FaGi, Gi]
Fulton's conjecture is true for $\mathfrak{s}_n$-equivariant divisors on $\mo$
for $n\leq 24$.
\end{theorem}

Any vital curve $C$ is represented by a partition $a+b+c+d=n$.
Specifically the n-pointed stable curve corresponding to the generic
point of $C$ has n-3 components and therefore exactly one component
with 4 'special points'-marked points or nodes-which partition the marked points into four sets.  Figure 1, for example, corresponds to the partition $1+2+3+4=10$.
 The intersections with such a curve are well known:

\begin{theorem}[Keel, McKernan][KeM]\label{number}
If a divisor equals $\sum_{j=2}^{\lfloor n/2\rfloor}r_j D_j$ than
it is F-nef if and only if
$$r_{a+b}+r_{a+c}+r_{a+c}-r_a-r_b-r_c-r_d \geq 0$$ for all partitions
$a+b+c+d=n$.  Here $r_i=r_{n-i}$ and $r_1=0$.
\end{theorem}

We are now in position to prove our main result.

\begin{theorem}[Main Theorem]\quad\\
Fix $n \ge 4$ and $\alpha$ a rational number in $(\frac{2}{n-1},1]$.
Assume the ($\mathfrak{s}_n$-equivariant) Fulton conjecture.  Let
$\overline{M}_{0,n}(\alpha)$ denote the log canonical model
of $\overline{M}_{0,n}$ with respect to $K+\alpha D$.
If $\alpha$ is in the range $(\frac{2}{k+2},\frac{2}{k+1}]$
for some $k=1,..., \lfloor \frac{n-1}{2} \rfloor$  then
         $\overline{M}_{0,n}(\alpha) \simeq \overline{M}_{0,(1/k,...,1/k)}$.
If $\alpha$ is in the range $(\frac{2}{n-1},\frac{2}{\lfloor n/2\rfloor + 1}]$ then
         $\overline{M}_{0,n}(\alpha) \simeq (P^1)^n // SL_2.$

\end{theorem}

Applying corollary \ref{cor} we will prove the main theorem by showing that the $A_{\alpha}$ divisors are F-nef.

\begin{proposition}\label{main}
$A_{\alpha}$ is F-nef for all $n,\alpha $.
\end{proposition}
Using the above criterion, we need only check that the coefficients
of $A_{\alpha}$ satisfy the inequalities for each partition.  We
observe that if $\alpha$ is in \\$[2/(k+2),2/(k+1)]$ then
$A_{\alpha}$ is a convex sum of $A_{2/(k+2)}$ and $A_{2/(k+2)}$:
$$\alpha =\frac{2t}{k+2}+\frac{2(t-1)}{k+1}\text{ implies }
A_{\alpha}=tA_{2/(k+2)} +(1-t)A_{2/(k+2)}.$$ Therefore to prove that
intersections are nonnegative we only need to check $\alpha$ of the form
$\frac{2}{k+1}$. Henceforth let $\alpha=\frac{2}{k+1}.$ By
theorems \ref{nef} and \ref{neff} we need only check
 $\alpha$ for  $2 \leq k \leq \lfloor n/2 \rfloor-1$.
\begin{lemma}
Let $n\geq 6$ and $1< k< \lfloor n/2 \rfloor$ be arbitrary. Set
$\alpha=\frac{2}{k+1}$ and define the following functions:
$$f(j)=\frac{j(j-1)\left((n-1)\frac{\alpha}{2}-1\right)}{n-1},  \quad  g(j)=\frac{(n-1)(\alpha+j-2)-j(j-1)}{n-1}.$$
Then, for any sum of positive integers $a+b+c+d=n$  the function
\begin{equation}
h(j)=
\begin{cases}
f(j),   &\text{for $1\leq j\leq k$;}\\
g(j),   &\text{for $k< j< n-k$;}\\
f(n-j),   &\text{for $n-k\leq j\leq n$}\\
\end{cases}
\end{equation}
\text{satisfies }$h(a+b)+h(a+c)+h(a+d)-h(a)-h(b)-h(c)-h(d)\geq 0.$
\end{lemma}

\begin{proof}
The function $h(a+b)+h(a+c)+h(a+d)-h(a)-h(b)-h(c)-h(d)$ depends
heavily on what the value of $k$ is with respect to $n,a,b,c$ and
$d$. For a given $n,a,b,c,d$ with $a\leq b\leq c\leq d$ we will
partition the possible values of $k$ into subsets on which $h$ acts
'nicely'. Specifically, we will look at the following partition into
cases. The fact that some of these cases may not exist for certain
values of $a,b,c,d$ will be dealt with as the individual cases are
examined.

\begin{enumerate}
\item $a>k$  and $a\leq b\leq c\leq d$ as always.
\item $a\leq k, b>k$
\item $b\leq k, c>k, a+b>k$
\item $b\leq k, c>k, a+b\leq k$
\item $c\leq k, d>k, a+b>k$
\item $c\leq k, d>k, a+b\leq k, a+c>k$
\item $c\leq k, d>k, a+c\leq k,b+c>k$
\item $c\leq k, d>k, b+c\leq k,a+b+c> k$
\item $c\leq k, d>k, b+c\leq k,a+b+c\leq k$
\item $d\leq k, a+b>k$
\item $d\leq k, a+b\leq k, a+c>k$
\item $d\leq k, a+c\leq k, a+d>k,     b+c>k$
\item $d\leq k, a+c\leq k, a+d\leq k, b+c>k$
\item $d\leq k, a+c\leq k, a+d> k,    b+c\leq k$
\item $d\leq k, a+c\leq k, a+d\leq k, b+c\leq k$
\end{enumerate}

These cases are all-inclusive.  We will now show that
$h(a+b)+h(a+c)+h(a+d)-h(a)-h(b)-h(c)-h(d)$ is nonnegative for each
of these subsets.  In all cases, the inequality comes easily from
the inequalities in the given cases.  We will skim the details since they are easy and repetitive.\\\\
\bold{case 1)} In this case, the inequalities show us that
$h(a+b)+h(a+c)+h(a+d)-h(a)-h(b)-h(c)-h(d)$ can be written as
$g(a+b)+g(a+c)+g(a+d)-g(a)-g(b)-g(c)-g(d)$.  This is because all
the values $a,b,c,d,a+b,a+c,a+d$ are by assumption greater than
$k$ and they are all less than $n-k$ by observation.  For example
$a+d\leq n-k$ since $b>k$ and $c>k$ and $a+b+c+d=n$. Using this we
get the calculation(after some simplification)
\begin{align*}
&\quad h(a+b)+h(a+c)+h(a+d)-h(a)-h(b)-h(c)-h(d)\\
&=g(a+b)+g(a+c)+g(a+d)-g(a)-g(b)-g(c)-g(d)\\
&=\frac{2k}{1+k}\end{align*}
 which is clearly positive.\\\\
\bold{case 2)} Under the given hypotheses we get
$h(a+b)+h(a+c)+h(a+d)-h(a)-h(b)-h(c)-h(d)=g(a+b)+g(a+c)+g(a+d)-f(a)-g(b)-g(c)-g(d)$.
This gives us
$$h(a+b)+h(a+c)+h(a+d)-h(a)-h(b)-h(c)-h(d)=\frac{a(2+k-a)}{1+k}$$
which is again nonnegative since $k\geq a$ in this case.\\\\
\bold{case 3)} In this case we have
\begin{align*}
&\quad h(a+b)+h(a+c)+h(a+d)-h(a)-h(b)-h(c)-h(d)\\
&=g(a+b)+g(a+c)+g(a+d)-f(a)-f(b)-g(c)-g(d)\\
&=\frac{(b - 2) (k - b) + a (2 + k - a)}{1+k}.\end{align*}  The
second half of this is positive since $k\geq a$.  The first half
is similarly positive unless $b=1$.  but that would imply $a=1$
and thus $2=a+b>k$ which does not hold.   Thus the whole equation
is positive.\\\\
\bold{case 4)} Here we have
\begin{align*}
&\quad h(a+b)+h(a+c)+h(a+d)-h(a)-h(b)-h(c)-h(d)\\
&=f(a+b)+g(a+c)+g(a+d)-f(a)-f(b)-g(c)-g(d)\\
&=\frac{2 a b}{1+k}>0.\end{align*}\\\\
 \bold{case 5)}
\begin{align*}
&\quad h(a+b)+h(a+c)+h(a+d)-h(a)-h(b)-h(c)-h(d)\\
&=g(a+b)+g(a+c)+g(a+d)-f(a)-f(b)-f(c)-g(d)\\
&=\frac{(k - c)(c - 2) + (k - b)(b - 2) + a(2 + k - a)}{1+k}\end{align*} which is positive except potentially when $b=1$.  However, $a=b=1$ is not possible in this case.\\\\
\bold{case 6)}
\begin{align*}
&\quad h(a+b)+h(a+c)+h(a+d)-h(a)-h(b)-h(c)-h(d)\\
&=f(a+b)+g(a+c)+g(a+d)-f(a)-f(b)-f(c)-g(d)\\
&=\frac{2 a b + (c - 2) (k - c)}{1+k}\end{align*} which is positive because $k\geq c$ and $a+c>k$ so that $c\geq 2.$\\\\
\bold{case 7)}
\begin{align*}
&\quad h(a+b)+h(a+c)+h(a+d)-h(a)-h(b)-h(c)-h(d)\\
&=f(a+b)+f(a+c)+g(a+d)-f(a)-f(b)-f(c)-g(d)\\
&=\frac{a (a + 2 b + 2 c - 2 - k)}{1+k}>0\text{ since } b+c>k.\end{align*}\\\\
 \bold{case 8)}
 \begin{align*}
&\quad h(a+b)+h(a+c)+h(a+d)-h(a)-h(b)-h(c)-h(d)\\
&=f(a+b)+f(a+c)+f(b+c)-f(a)-f(b)-f(c)-g(d)\\
&=\frac{(a+b+c-2)(a+b+c-k)}{1+k}> 0\text{ since } a+b+c>k.\end{align*}\\\\
 \bold{case 9)}
 \begin{align*}
&\quad h(a+b)+h(a+c)+h(a+d)-h(a)-h(b)-h(c)-h(d)\\
&=f(a+b)+f(a+c)+f(b+c)-f(a)-f(b)-f(c)-f(a+b+c)\\
&=0\end{align*}\\\\
\bold{case 10)}
 \begin{align*}
&\quad h(a+b)+h(a+c)+h(a+d)-h(a)-h(b)-h(c)-h(d)\\
&=g(a+b)+g(a+c)+g(a+d)-f(a)-f(b)-f(c)-f(d)\\
&=\frac{(k - d)(d - 2) + (k - c)(c - 2) + (k - b)(b - 2) + a(2 + k - a)}{1+k}\end{align*} is positive by the identical argument as in case $5$.\\\\
 \bold{case 11)}
 \begin{align*}
&\quad h(a+b)+h(a+c)+h(a+d)-h(a)-h(b)-h(c)-h(d)\\
&=f(a+b)+g(a+c)+g(a+d)-f(a)-f(b)-f(c)-f(d)\\
&=\frac{2 a b + (c - 2) (k - c) + (d - 2)(k - d)}{1+k}> 0\text{ since } a+c>k,d\leq k.\end{align*}\\\\
 \bold{case 12)}
 \begin{align*}
&\quad h(a+b)+h(a+c)+h(a+d)-h(a)-h(b)-h(c)-h(d)\\
&=f(a+b)+f(a+c)+g(a+d)-f(a)-f(b)-f(c)-f(d)\\
&=\frac{a (2 b + 2 c + a - 2 - k) + (d - 2) (k - d)}{1+k}> 0\text{ since } b+c>k,d\geq 2, k\geq d.\end{align*}\\\\
 \bold{case 13)}
 \begin{align*}
&\quad h(a+b)+h(a+c)+h(a+d)-h(a)-h(b)-h(c)-h(d)\\
&=f(a+b)+f(a+c)+f(a+d)-f(a)-f(b)-f(c)-f(d)\\
&=\frac{2 a (n - 2 - k)}{1+k}>0.\end{align*}\\\\
 \bold{case 14)}
 \begin{align*}
&\quad h(a+b)+h(a+c)+h(a+d)-h(a)-h(b)-h(c)-h(d)\\
&=f(a+b)+f(a+c)+f(b+c)-f(a)-f(b)-f(c)-f(d)\\
&=\frac{(n-2d)(n-2)}{1+k}>0\text{ since }d\leq k<\floor.\end{align*}\\\\
 \bold{case 15)}
This last case can't actually happen with the assumption $k<\floor$ .\\
This concludes the proof of our lemma.
\end{proof}
We note that in our intersection calculation we got a strictly
positive number in every case except when $a+b+c\leq k$.  This is
exactly what we would expect since those values represent precisely
the curves which get contracted under the map $\rho:\mo\rightarrow
\wo$.
\begin{proof}[Proof of Proposition \ref{main}]\quad\\
By applying Theorem \ref{number} we have proved that $A_{\alpha}$
intersects nonnegatively with all vital curves in the cases when
$\alpha=\frac{2}{k+1}$, $2\leq k\leq \lfloor n/2 \rfloor-1$ and
$n\geq 6$.  Since divisors that are nef are certainly also F-nef,
 we have shown that $A_{\alpha}$ is F-nef for all $\frac{2}{n-1}\leq \alpha \leq 1$, $n\geq 6.$ The result is trivial for $n<6$.

\end{proof}

\section{Partial results}
Certain ranges of $A_{\alpha}$ can be explicitly shown to be nef.

\begin{theorem}\label{nef}
For $\frac{2}{n-1}< \alpha\leq\frac{2}{\lfloor n/2 \rfloor+1},
A_{\alpha}$ is semi-ample.

\end{theorem}

\begin{proof}
This range corresponds exactly to when our map $\rho$ is of the form
$\rho: \mo\to {\IP}^1/SL_2$.  Because we have contracted all the
rays corresponding to $D_j$, $j\neq $2 of $\mo$, the
$\mathfrak{s}_n$-equivariant Picard group of ${\IP}^1/SL_2$ is one dimensional
and generated by the image of $D_2.$  Also the canonical class is
simply $$K={\rho}_*(K_{\mo})={\rho}_*(\frac{-2}{n-1}D_2+\rho-\text{exceptional divisors})=\frac{-2}{n-1}[\text{image of }D_2].$$  Therefore
$K+\alpha E =(\alpha-\frac{2}{n-1})$[image of $D_2$] is ample for
our range of $\alpha$ which implies $A_{\alpha}$ is the pullback of an
ample divisor as we wanted.

\end{proof}

\begin{theorem}[KM]\label{neff}
$A_{\alpha}$ is ample for $2/3 < \alpha \leq 1$
\end{theorem}
Keel and McKernan actually prove this for $A_1=K_{\mo}+D$.
However $A_{\alpha} =K_{\mo}+\alpha D$ in these cases and the proof generalizes word for word.\\
Keel and McKernan also developed a combinatorial tool that helps
prove in special cases that an F-nef divisor is actually nef.

\begin{theorem}[KeM]
Let R be an extremal ray of the cone $\bar{NE}_1$.  Supposed D is a
Q-divisor with $D-G$ effective and $(K_{\mo}+G).R\leq 0$.  Then $R$
is spanned by a vital curve.

\end{theorem}

\begin{corollary}
If $K_{\mo}+G$ as above is F-nef then it is nef.
\end{corollary}
\begin{proof}
Otherwise, we'd have some R with $(K_{\mo}+G).R< 0$.  But that would
imply that $K_{\mo}+G$ intersects a vital curve negatively which is
a contradiction.

\end{proof}

\begin{corollary}
$A_{\alpha}$ is nef for $\alpha\geq \frac{1}{3}.$

\end{corollary}

Let $\alpha=\frac{2}{k+1}$ and $n >
 13$.
\begin{proof}

It is enough to show that some positive multiple of $A_{\alpha}$ is
nef. Let $c>0$ and set $B=cA_{\alpha}-K_{\mo}=$
$$\sum_{j=1}^{k}D_j\left(\genfrac(){0cm}{0}j 2 c \alpha
-(j-2)-\frac{j(j-1)}{n-1}\right)+\sum_{j=k+1}^{\lfloor
n/2\rfloor}D_j\left( c(j-2+\alpha)-(j-2)-\frac{j(j-1)}{n-1}\right)$$
If we can find a value of $c$ so that each coefficient is between
zero and one, then we can apply
the proposition.\\

By examining the lower terms we immediately are forced to have the
inequalities:

$$c\geq\frac{k+1}{n-1}+\frac{k+1}{6}\text{ \quad and \quad }
c\leq\frac{k+1}{n-1}+\frac{k+1}{k}.$$

Therefore this only has a chance of satisfying our condition for
$k\leq 6$.\\
 Similarly we get the following inequality from bounds on
the upper terms:
$$c\geq\frac{(\floor-2)(n+\floor)+2}{(n-1)(\floor-2+\alpha)}\text{ \quad and \quad
}c\leq\frac{k(n-k)}{(n-1)(k-1+\alpha)}.$$

For $k\leq 6$ we can simply plug $k$ into these 4 inequalities and
we see that they are all satisfied for $k\leq 5$ but not for $k=6$.

For higher values of $k$ we may at least hope to make all terms
positive and some terms $\leq 1$.  We can do this, but the
inequalities still force most coefficient of the $D_j$'s to be quite
large.
\end{proof}

\section{The Fulton Cone}

 One way to prove that our
$A_{\alpha}$'s are nef would be to prove Fulton's conjecture.  We of
course don't need quite so strong a result and in fact, from
computational data, it appears that at least some of the
$A_{\alpha}$'s sit in a very nice part of the Fulton cone of F-nef
divisors.  Although the major combinatorial techniques don't appear
to be of use, it would be nice if this corner of the Fulton cone
could be proven to be nef by geometric means.

\begin{notation}
Let W=$\IP(\IQ^{\floor-1})$ be the projectivized cone of
$\mathfrak{s}_n$-equivariant divisors of $\bar{\bo{M}}_{0,n}$ with standard
basis $D_2,\dots D_{\floor}$ and corresponding coordinates
$(r_2\ldots r_{\floor})$.  Let $F$ be the subcone of F-nef divisors.

For each sum $a+b+c+d=n$ we have an associated hyperplane in $W$. We
will denote the 'special' hyperplanes corresponding to
$1+1+i+(n-i-2)=n$ by $V_i$.  Other hyperplanes will
simply be denoted $V(a,b,c,d)$.
\end{notation}
 We will be studying the portion of $F$ which contains $A_{\alpha}$ for small $\alpha.$

\begin{proposition}\label{face}
For \small$2\leq k\leq \floor$,\normalsize the minimal face of $F$ containing $A_{\frac{2}{k+1}}$ is $$F_{k}=V_1\intersect V_2\intersect\dots\intersect V_{k-2}\intersect F\subset W.$$  In particular $A_{\frac{2}{k+1}}$ is in the interior of a face of (projective) dimension $\floor-k$ whose boundary contains $A_{\alpha}$ for $\alpha \leq \frac{2}{k+2}.$

\end{proposition}
Here $F_2=F$ and $F_{\floor}$ is a (projective) vertex of $F$ equal to $A_{\frac{2}{\floor+1}}$ (which is in fact nef unconditionally by Theorem \ref{nef}).
\begin{proof}
By the proof of Theorem \ref{main} $A_{\frac{2}{k+1}}$ contracts precisely vital curves corresponding to partitions $a+b+c+d=n$ such that $a+b+c\leq k$.  Therefore the minimal face of $F$ containing $A_{\frac{2}{k+1}}$ is cut out by $\bigcap_{a+b+c\leq k}V(a,b,c,d).$\\

\textbf{Claim:}
$$\bigcap_{a+b+c\leq k}V(a,b,c,d)=V_1\intersect V_2\intersect\dots\intersect V_{k-2}.$$
We certainly have $\bigcap_{a+b+c\leq k}V(a,b,c,d)\subset V_1\intersect V_2\intersect\dots\intersect V_{k-2}$.  To prove the other inclusion we give an explicit description of $V_1\intersect V_2\intersect\dots\intersect V_{k-2}.$

$V_1\intersect V_2\intersect\dots\intersect V_{k-2}$ is defined by equations $r_2+2r_{i+1}-r_i-r_{i+2}=0$, or $r_{i+2}=r_2+2r_{i+1}-r_i$ for $i\leq k-2.$  The only solutions up to multiplication by a constant for this recurrence relation are $r_i=\genfrac(){0cm}{0}j 2$ for $2\leq i\leq k$ and $r_i$ free for $k<i\leq\floor$.

 Now for $a+b+c\leq k\leq\floor$,  $V(a,b,c,d)$ has defining equation $r_{a+b}+r_{a+c}+r_{b+c}-r_a-r_b-r_c-r_{a+b+c}=0$. To prove the claim we need only show that any such equation is satisfied by all points of $W$ with $r_i=\genfrac(){0cm}{0}j 2$ for $ i\leq k$.  This condition is equivalent to ${\genfrac(){0cm}{0}{a+b} 2}+{\genfrac(){0cm}{0}{a+c} 2}+{\genfrac(){0cm}{0}{b+c} 2}-{\genfrac(){0cm}{0}{a} 2}-{\genfrac(){0cm}{0}{b} 2}-{\genfrac(){0cm}{0}{c} 2}-{\genfrac(){0cm}{0}{a+b+c} 2}=0$ which is always true.  This proves our claim.

 The claim immediately implies the first statement of the proposition.  Additionally, we showed in the proof of the claim that the dimension of $V_1\intersect V_2\intersect\dots\intersect V_{k-2}$ is $(\floor-2)-(k-2)=\floor-k.$  Since each $A_{\alpha}$ is in $F$, this completes the proof.

\end{proof}
The following is a description of the faces $F_{\floor-1}$ and $F_{\floor-2}$.  Specifically, for each $n$, we make a table with the following information: the first column lists the hyperplanes which cut out the vertices of the (1-dimensional) face $F_{\floor-1}$, the second column describes the shape of $F_{\floor-2}$ and the last column lists the hyperplanes which cut out edges of $F_{\floor-2}$.\\\\
\begin{tabular}{| l | l | l | l |}
\hline n  &  facet inducing & \# of facets  &
 facet inducing \\
   & hyperplanes of $F_{\floor-1}$ & of $F_{\floor-2}$ &
  hyperplanes of $F_{\floor-2}$

   \\ \hline 6   &  $V_1$ and $V_2$  & & \\
\hline 7   & $V_1$ and $V(1,2,2,2)$ & & \\ \hline 8   & $V_2$ and
$V_3$ &4 & $V_1, V_2, V_3, V(2,2,2,2)$ \\ \hline 9   & $V_2$ and
$V_3$ &4 & $V_1, V_2, V_4, V(2,2,2,3)$ \\ \hline 10   & $V_3$ and
$V_4$ &4 & $V_2, V_3, V_4, V(1,3,3,3)$ \\ \hline 11   & $V_3$ and
$V_4$ &4 & $V_2, V_3, V_4, V(2,3,3,3)$ \\ \hline 12   & $V_4$ and
$V_5$ &3 & $V_3, V_4, V_5$ \\ \hline 13   & $V_4$ and $V_5$ &4 &
$V_3, V_4, V_5, V(1,4,4,4)$ \\ \hline $n=14,\dots$   &
$V_{\lfloor{n/2}\rfloor-2}$, $V_{\lfloor n/2\rfloor-1}$ &3 &
$V_{\lfloor n/2\rfloor-3}, V_{\lfloor n/2\rfloor-2} ,V_{\lfloor
n/2\rfloor-1}$ \\ \hline
\end{tabular}
\begin{example}
For $n=8$, the faces $F_{\floor}$ and $F_{\floor-1}$ have already stabilized into the form that all such faces have for large $n$. $F_{\floor-2}$, on the hand, has not yet stabilized into the expected form of a 3-simplex.

\begin{figure}[h]
  \begin{center}
    \includegraphics{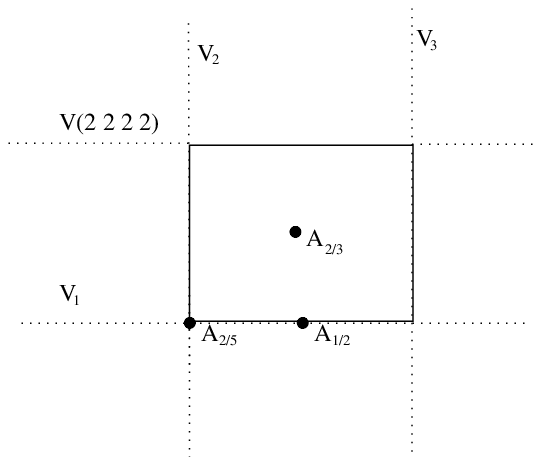}
  \end{center}
  \caption{\quad The cone $F$ for ${\bar{\bo{M}}_{0,8}}$}
\end{figure}

\end{example}
 The structure of vertices, edges, etc. of the higher
dimensions faces $F_{\floor-t}$ stabilize into a
similar pattern. In fact we can make this more precise.\\
\begin{theorem}[Version 1]
The Fulton cone is stably simplicial in the neighbor of the
$A_{\alpha}$'s in the sense that, given $t$,  $F_{\floor-t}$ is a simplex for $n>>0$.
\end{theorem}

 \begin{definition}
 By the \emph{F-simplex} we denote the simplex in $W$ induced by $V_1,\dots,V_{\floor-1}.$  Since the equations for these hypersurfaces are linearly independent this does indeed make a simplex.  We denote the vertices by $p_{k+1}=\bigcap_{i\neq k}V_i.$
 \end{definition}

By Proposition \ref{face} we know that $F_k$ is contained in the face of the F-simplex cut out by $V_1\intersect V_2\intersect\dots\intersect V_{k-2}$.  Such a face is of course again a simplex of smaller dimension.  To prove the theorem we will show that this simplex is actually equal to $F_{k}$ for sufficiently large $k$.  It is enough to prove that the vertices $p_k$ are $F-nef$ for large $k$.

\begin{lemma}
The point $p_{k}$ is represented by the divisor \begin{align*}&\frac{2}{n-1}\left[\sum_{i=2}^k {\genfrac(){0cm}{0}{i} 2} D_i+\sum_{i=k+1}^{\floor}\left( {\genfrac(){0cm}{0}{i} 2}-(i-k)\left(\frac{n-1}{2}\right)\right)D_i\right]\\
&=-K_{\bar{\bo{M}}_{0,n}}+ \sum_{i=2}^k
(i-2)D_i+\sum_{i=k+1}^{\floor}(k-2)D_i\end{align*}
\end{lemma}

\begin{proof}[Proof of Lemma]
The equation of $V_i$ for $i\leq\floor-2$ is $$r_{i+2}=r_2+2r_{i+1}-r_i.$$  Ignoring $V_{\floor-1}$ for the moment, we can solve these equations for \quad\quad\quad\quad \quad $2\leq i\leq k-2, k\leq i\leq \floor-2$.  This gives us: $$r_i={\genfrac(){0cm}{0}{i} 2}\text{ for }i\leq k;$$  $$r_i=\left({\genfrac(){0cm}{0}{i-k} 2}+(i-k)r_{k+1}-(i-k-1){\genfrac(){0cm}{0}{k} 2}\right)\text{ for }k+1\leq i\leq \floor.$$  Depending on the parity of $n$, $V_{\floor-1}$ is cut out by $r_{n/2}=r_{n/2-1}-\frac{1}{2}$ or $r_{\floor}=r_{\floor-1}-1.$  We solve for $r_{k+1}$ by plugging in the values for $r_{\floor}$ and $r_{\floor-1}.$  Doing so in either the even or the odd case gives us $r_{k+1}={\genfrac(){0cm}{0}{k+1} 2}-\frac{n-1}{2}.$  This in turn implies that for $i\geq k+1$ we have $$r_i={\genfrac(){0cm}{0}{i-k} 2}+(i-k)({\genfrac(){0cm}{0}{k+1} 2}-\frac{n-1}{2})-(i-k-1){\genfrac(){0cm}{0}{k} 2}={\genfrac(){0cm}{0}{i} 2}-(i-k)\left(\frac{n-1}{2}\right).$$
\end{proof}
We can now explicitly compute the vertices of the F-simplex which are F-nef.

\begin{theorem}[Version 2]
$F_k$ is the face of the F-simplex spanned by $p_{k},\dots p_{\floor}$ whenever $k\geq\frac{n}{3}$.  This bound is tight.
\end{theorem}
\begin{proof}
As usual we will look at curves $C$ corresponding to a partition $a+b+c+d=n$ with $a\leq b\leq c\leq d.$\\\\
\bold{case 1)} $a+b \geq k$.\\
In this case we have $p_{k}\text{\Huge.\normalsize}C =$
$$-2 + (k-2) + (k-2)+(k-2)-\min(a-2,k-2) -\min(b-2,k-2) -\min(c-2,k-2) -\min(d-2,k-2).$$
Since $k\geq n/3$ this implies $a,b\leq k$, so in fact (canceling
the two's) we have
$$p_{k}\text{\Huge.\normalsize}C = 3k -a -b -\min(c,k) -\min(d,k)\geq 3k-a-b-c-d=3k-n \geq 0.$$\\\\
 \bold{case 2)} $a+d\leq k$.\\
by canceling terms we get $$p_{k}\text{\Huge.\normalsize}C =(a+b)+(a+c)+(a+d)-a-b-c-d=2a>0.$$
 \bold{case 3)} $a+b<k, a+d> k$.\\
 $$p_{k}\text{\Huge.\normalsize}C =
 (a+b)+\min(a+c,k)+\min(k,b+c)-a-b-\min(c,k)-\min(d,k,a+b+c)$$$$=\min(a+c,k)+\min(k,b+c)-\min(c,k)-\min(d,k,a+b+c)\geq 0$$

These three cases are all-inclusive so we have that $p_{k}$ is F-nef
whenever $k\geq\frac{n}{3}$. For tightness, we refer to the
following example.
\end{proof}
\begin{example}
Let $n=3*l+p$ where $p\in\{1,2,3\} ,l\geq 2$.  We will look at
$p_{\floor-k}$ where $k=l$.  Let $C$ be represented by $p+l+l+l=n$.
Then $_{\floor-k}\text{\Huge.\normalsize}C = -1+3(l-2)-3(k-2)=-1<0$.
Thus $p_{\floor-k}$ is not F-nef is this case.

\end{example}

\begin{conjecture}[Weak Fulton]
The face of the $\mathfrak{s}_n$-equivariant nef cone containing $A_{\alpha}$ for $\alpha\leq\frac{2}{[n/3]+1}$ is a simplex generated  by $p_{[n/3]},\dots p_{\floor}$.

\end{conjecture}

\begin{remark}
The weak Fulton conjecture would not by itself prove our description of the log canonical models of $\mo$.  It would however give us a strong partial result.

\end{remark}
\section{References}
[BCHM] C. Birkar, P. Cascini, C. Hacon, J. McKernan, \emph{Existence
of minimal models for varieties of log general type},
math.AG/0610203\newline\newline
[CoHa] M. Cornalba, J. Harris, \emph{Divisor classes associated to families of stable varieties, with applications to the moduli space of curves},  Ann. Sci. �cole Norm. Sup. (4)  21  (1988),  no. 3, 455--475.\newline\newline
[CHS] I. Coskun,, J. Starr, J. Harris, \emph{The ample cone of the Kontsevich moduli space}, Canad. J. Math., to
 appear.\newline\newline
 [EH] D. Eisenbud J. Harris, \emph{The Kodaira dimension of
the moduli space of curves of genus $\geq 23$}, Invent. Math.  90
(1987), no. 2, 359-387.\newline
\newline
 [Fa] G. Farkas, \emph{Syzygies of curves
and the effective cone of $\bar{M}_g$}, Duke Math. J.  135  (2006),
no. 1, 53--98. \newline
\newline
[FaGi] G. Farkas, A. Gibney, \emph{The Mori cones of moduli spaces
of pointed curves of small genus}, Trans. Amer. Math. Soc.  355
(2003),  no. 3, 1183--1199.\newline\newline
 [Gi] A. Gibney, \emph{Numerical criteria for divisors on $M_{g}$ to be ample},
 math.AG/0312072.\newline\newline
  [GKM] A. Gibney, S. Keel, I. Morrison,
\emph{Towards the ample cone of $\overline M\sb {g,n}$,} J. Amer.
Math. Soc. 15 (2002), no. 2, 273--294.\newline\newline
 [HaMu] J.
Harris, D. Mumford, \emph{On the Kodaira dimension of the moduli
space of curves}, Invent. Math. 67 (1982), no. 1,
23-88.\newline\newline
 [Ha] B. Hassett, \emph{Moduli spaces of
weighted pointed stable curves}, Adv. Math.  173  (2003),  no. 2,
316--352.\newline\newline
[HaMo] J. Harris, I. Morrison,
\emph{Slopes of effective divisors on the moduli space of stable
curves},  Invent. Math.  99  (1990),  no. 2,
321--355.\newline\newline
 [HH] B. Hassett, D. Hyeon,
\emph{Log canonical models for the moduli space of curves: First
divisorial contraction},math.AG/0607477v1.\newline\newline
[Ke] S. Keel,\emph{ Intersection theory of
moduli space of stable $n$-pointed curves of genus zero}.  Trans.
Amer. Math. Soc.  330  (1992),  no. 2, 545--574.\newline\newline
 [KeM] S. Keel, J. McKernan, \emph{Contractible Extremal Rays on
 $\overline{M}_{0,n}$}, math.AG/9607009.\newline\newline
 [KM] J. Kollar, S. Mori. \emph{Birational geometry of algebraic varieties}, Cambridge Tracts in Mathematics, 134. Cambridge University Press, Cambridge, 1998.\newline\newline
 [Pan] R. Pandharipande,
\emph{The canonical class of $\overline{M}\sb {0,n}(\bold P\sp r,d)$ and enumerative geometry},
Internat. Math. Res. Notices (1997), no. 4, 173--186.
\end{document}